\documentclass[reqno]{amsart}
\usepackage{amsmath, amsthm, amssymb, amstext}

\usepackage{hyperref,xcolor}
\hypersetup{
 pdfborder={0 0 0},
 colorlinks,
}
\usepackage{enumitem}
\newtheorem{theorem}{Theorem}
\newtheorem{remark}[theorem]{Remark}
\newtheorem{lemma}[theorem]{Lemma}
\newtheorem{proposition}[theorem]{Proposition}
\newtheorem{corollary}[theorem]{Corollary}

\newtheorem{example}[theorem]{Example}

\DeclareMathOperator*{\divergenz}{div}              %
\DeclareMathOperator*{\esssup}{ess ~sup}         %
\DeclareMathOperator*{\Ss}{S}

\newcommand{\N}{\mathbb{N}}

\newcommand{\R}{\mathbb{R}}

\newcommand{\Lp}[1]{L^{#1}(\Omega)}
\newcommand{\Lpvalued}[1]{L^{#1}\left(\Omega;\R^N\right)}
\newcommand{\Lprand}[1]{L^{#1}(\partial\Omega)}
\newcommand{\Wp}[1]{W^{1,#1}(\Omega)}

\newcommand{\ran}{\rangle}
\newcommand{\eps}{\varepsilon}
\newcommand{\ph}{\varphi}
\newcommand{\Om}{\Omega}
\newcommand{\rand}{\partial\Omega}

\newcommand{\into}{\int_{\Omega}}

\newcommand{\Linf}{L^{\infty}(\Omega)}
\newcommand{\close}{\overline{\Omega}}

\newcommand{\cprime}{$'$}

\renewcommand{\l}{\left}
\renewcommand{\r}{\right}
\numberwithin{theorem}{section}
\numberwithin{equation}{section}

\title[Moser iteration applied to elliptic equations with critical growth]{Moser iteration applied to elliptic equations with critical growth on the boundary}

\author[G.\,Marino]{Greta Marino}
\address[G.\,Marino]{Dipartimento di Matematica e Informatica, Universit\`{a} degli Studi di Catania, Viale A.\,Doria 6,I-95125 Catania, Italy}
\email{greta.marino@dmi.unict.it}

\author[P.\,Winkert]{Patrick Winkert}
\address[P.\,Winkert]{Technische Universit\"{a}t Berlin, Institut f\"{u}r Mathematik, Stra\ss e des 17.\,Juni 136, 10623 Berlin, Germany}
\email{winkert@math.tu-berlin.de}

\subjclass[2010]{35J60, 35B45, 35J25}
\keywords{Moser iteration, boundedness of solutions, a priori bounds, elliptic operators of divergence type, critical growth on the boundary}
                  
\begin{document}

\begin{abstract}
    This paper deals with boundedness results for weak solutions of the equation
    \begin{align}\label{problem0}\tag{P}
	\begin{aligned}
	    -\divergenz \mathcal{A} (x,u,\nabla u) & = \mathcal{B}(x,u,\nabla u)  & \hspace*{0.5cm} & \text{ in } \Om, \\
	\mathcal{A} (x,u,\nabla u)\cdot \nu & = \mathcal{C}(x,u)  && \text{ on } \partial \Omega,
	\end{aligned}
    \end{align}
    where the functions $\mathcal{A}: \Omega \times \R \times \R^N \to \R^N$, $\mathcal{B}: \Omega \times \R \times \R^N \to \R$, and $\mathcal{C}: \partial \Omega \times \R \to \R$ are Carath\'eodory functions satisfying certain $p$-structure conditions that have critical growth even on the boundary. Based on a modified version of the Moser iteration we are able to prove that every weak solution of \eqref{problem0} is bounded up to the boundary. Under some additional assumptions on the functions $\mathcal{A}$ and $\mathcal{C}$ this leads directly to regularity for weak solutions of \eqref{problem0}. 
\end{abstract}

\maketitle

\section{Introduction}

Let $\Omega \subset \R^N$, $N>1$, be a bounded domain with a Lipschitz boundary $\partial \Omega$. In this paper, we study the boundedness of weak solutions of the problem
\begin{align}\label{problem}
    \begin{aligned}
        -\divergenz \mathcal{A} (x,u,\nabla u) & = \mathcal{B}(x,u,\nabla u)  & \hspace*{0.5cm} & \text{ in } \Om, \\
    \mathcal{A} (x,u,\nabla u)\cdot \nu & = \mathcal{C}(x,u)  && \text{ on } \partial \Omega,
    \end{aligned}
\end{align}
where $\nu(x)$ denotes the outer unit normal of $\Omega$ at $x\in \partial \Omega$, and $\mathcal{A}, \mathcal{B}$ and $\mathcal{C}$ satisfy
suitable $p$-structure conditions, see hypotheses (H) in Section \ref{section_3}.

The main goal of this paper is to present a priori bounds for weak solutions of equation \eqref{problem}, where we allow critical growth on the data in the domain and on the boundary. The main idea in the proof is based on a modified version of Moser's iteration which in turn is based on the books of Dr{\'a}bek-Kufner-Nicolosi \cite{Drabek-Kufner-Nicolosi-1997} and Struwe \cite{Struwe-2008}. 

In some sense, \eqref{problem} is a generalization of the classical differential equation from the Yamabe problem
\begin{equation}
\label{yamabe}
-\Delta u= f(x) u+ h(x) u^{\frac{N+2}{N-2}},
\end{equation}
where $ f $ and $ h $ are smooth functions. It is well known that there is no stable regularity theory for solutions of equation \eqref{yamabe}, which reflects the difficulty of the Yamabe problem. Nevertheless, it was proven by Trudinger \cite{Trudinger-1968} that any solution $ W^{1,2}(\Omega) $ of \eqref{yamabe} is in fact smooth, but the regularity estimates depend on the solution itself. In this spirit, our main result, Theorem \ref{main_theorem}, can thus be seen as a generalization of Trudinger's work. 

The main novelty of our paper consists in the generality of the assumptions
needed to establish the boundedness of weak solutions to \eqref{problem}. In particular, the
assumptions on the nonlinearity $\mathcal{C}$ are rather general allowing critical growth on the boundary. To the best of our knowledge, such a treatment with critical growth even on the boundary has not been published before. Another novelty is a result of independent interest which shows that a Sobolev function, which is bounded in the domain, is also bounded on the boundary, see Proposition \ref{proposition_boundedness_boundary}.

Recently, Papageorgiou-R\u adulescu \cite[Proposition 2.8]{Papageorgiou-Radulescu-2016} studied a priori bounds for problems of the form
\begin{align}\label{problem11}
    \begin{aligned}
        -\divergenz a (\nabla u) & = f_0(x,u)  & \hspace*{0.5cm} & \text{ in } \Om, \\
    a (\nabla u)\cdot \nu & = -\beta |u|^{p-2}u  && \text{ on } \partial \Omega,
    \end{aligned}
\end{align}
where $1<p<\infty$, the function $a:\R^N\to\R^N$ is continuous and strictly monotone satisfying certain regularity and growth conditions, the Carath\'eodory function $f_0:\Omega\times \R\to\R$ has critical growth with respect to the second variable and $\beta \in C^{1,\alpha}(\partial \Omega)$ with $\alpha \in (0,1)$ and $\beta\geq 0$. Note that our setting is more general than those in \cite{Papageorgiou-Radulescu-2016} since we have weaker conditions on $a$ and $f_0$ and our boundary term is able to have critical growth. The proof of their result is mainly based on a treatment of Garc\'\i a Azorero-Peral Alonso \cite{Garcia-Azorero-Peral-Alonso-1994}, who studied equation \eqref{problem11} with the $p$-Laplacian and homogeneous Dirichlet condition, namely
\begin{align*}
    \begin{aligned}
        -\Delta_p u & = \lambda |u|^{q-2}u+|u|^{p^*-2}u  & \hspace*{0.5cm} & \text{ in } \Om, \\
    u & = 0  && \text{ on } \partial \Omega,
    \end{aligned}
\end{align*}
with $1<q<p<N$, $\lambda>0$ and the Sobolev critical exponent $p^*$, see Section \ref{section_2} for its definition. Both works use a different technique than the Moser iteration applied in our paper. For the semilinear case we mention the work of Wang \cite{Wang-1991}.

An alternative approach was published by Guedda-V\'eron \cite{Guedda-Veron-1989} who studied quasilinear problems for positive solutions given by
\begin{align*}
    \begin{aligned}
        -\Delta_p u & = a(x) u^{p-1}+u^{p^*-1}& \hspace*{0.5cm} & \text{ in } \Om, \\
        u& > 0& \hspace*{0.5cm} & \text{ in } \Om,\\
    u & = 0  && \text{ on } \partial \Omega,
    \end{aligned}
\end{align*}
with $a\in \Linf$. In all these works the assumptions on the functions are stronger than ours and no critical growth on the boundary is allowed.

Finally, we mention some works concerning boundedness and regularity results of weak solutions to quasilinear equations of the form \eqref{problem} that have subcritical growth, see, for example, Fan-Zhao \cite{Fan-Zhao-1999}, Gasi{\'n}ski-Papageorgiou \cite{Gasinski-Papageorgiou-2011}, \cite[pp.\,737--738]{Gasinski-Papageorgiou-2006}, Hu-Papageorgiou \cite{Hu-Papageorgiou-2011}, L{\^e} \cite{Le-2006}, Motreanu-Motreanu-Papageorgiou \cite{Motreanu-Motreanu-Papageorgiou-2009}, Pucci-Servadei \cite{Pucci-Servadei-2008}, Winkert \cite{Winkert-2010}, \cite{Winkert-2010c}, \cite{Winkert-2014}, Winkert-Zacher \cite{Winkert-Zacher-2012}, \cite{Winkert-Zacher-2015}, \cite{Winkert-Zacher-2016} and the references therein. The methods used in these papers are mainly based on Moser's iteration or De Giorgi's iteration technique and no critical growth occurs.

The paper is organized as follows. In Section \ref{section_2} we present the main preliminaries including a multiplicative inequality estimating the boundary integrals and a result how $\Linf$-boundedness implies $L^\infty(\partial \Omega)$-boundedness. In Section \ref{section_3} we state our main result and the proof is divided into several parts. First we prove that every weak solution belongs to $L^q(\Omega)$ for any $q<\infty$, then we show its belonging to $L^q(\partial \Omega)$ for any finite $q$. In the second part of the proof we consider the uniform boundedness and show that a weak solution belongs to $\Linf$ and $L^\infty(\partial\Omega)$, respectively. Finally, as an important application, we give general conditions on the functions $\mathcal{A}$ and $\mathcal{C}$ when a solution lies in $C^{1,\beta}(\close)$ for some $\beta \in (0,1)$ based on the regularity results of Lieberman \cite{Lieberman-1991}.

\section{Preliminaries}\label{section_2}

Let $r$ be a number such that $1\leq r<\infty$. We denote by $\Lp{r}$, $\Lpvalued{r}$ and $\Wp{r}$ the usual Lebesgue and Sobolev spaces equipped with the norms $\|\cdot\|_r$ and $\|\cdot\|_{1,r}$ given by
\begin{align*}
    \|u\|_r&=\left(\into |u|^rdx\right)^{\frac{1}{r}},\qquad \|\nabla u\|_{r}=\left(\into |\nabla u|^r dx\right)^{\frac{1}{r}}\\
    \|u\|_{1,r}&=\left(\into |\nabla u|^r dx+\into |u|^rdx\right)^{\frac{1}{r}}.
\end{align*}
For $r=\infty$ we recall that the norm of $\Linf$ is given by
\begin{align*}
    \|u\|_\infty=\esssup_\Omega |u|.
\end{align*}
On the boundary $\partial \Omega$, we use the $(N-1)$-dimensional Hausdorff (surface) measure denoted by $\sigma$. Then, in a natural way we can define the Lebesgue spaces $L^s(\partial \Omega)$ with $1\leq s \leq \infty$ and the norms $\|\cdot\|_{s,\partial \Omega}$ which are given by
\begin{align*}
    \|u\|_{s,\partial \Omega}=\left(\int_{\partial\Omega} |u|^sd\sigma\right)^{\frac{1}{s}}\quad (1\leq s<\infty),\qquad \|u\|_{\infty,\partial\Omega}=\esssup_{\partial\Omega}|u|.
\end{align*}
It is well known that there exists a unique linear continuous map $\gamma: \Wp{p} \to \Lprand{p_*}$ known
as the trace map such that $\gamma(u)=u\big|_{\partial \Omega}$ for all $u\in\Wp{p}\cap C(\close)$, where $p_*$ is the critical exponent on the boundary given by
\begin{align}\label{critical_exponent_boundary}
    p_*=
    \begin{cases}
	\frac{(N-1)p}{N-p} &\text{if }p<N,\\
	\text{any }q\in(1,\infty)&\text{if }p\geq N.
    \end{cases}
\end{align}
For the sake of notational simplicity, we drop the use of the trace map $\gamma$. It is understood that all restrictions of the Sobolev functions $u\in \Wp{p}$ on
the boundary $\partial \Omega$ are defined in the sense of traces.

Furthermore, the Sobolev embedding theorem guarantees the existence of a li\-near, continuous map $i: \Wp{p}\to \Lp{p^*}$ with the critical exponent in the domain given by
\begin{align}\label{critical_exponent_domain}
    p^*=
    \begin{cases}
	\frac{Np}{N-p} &\text{if }p<N,\\
	\text{any }q\in(1,\infty)&\text{if }p\geq N.
    \end{cases}
\end{align}
We refer to Adams \cite{Adams-1975} as a reference for the embeddings above. 

The norm of $\R^N$ is denoted by $|\cdot|$ and $\cdot$ stands for the inner product in $\R^N$. For $s \in \R$, we set $s^{\pm}=\max\{\pm s,0\}$ and for $u \in W^{1,p}(\Omega)$ we define $u^{\pm}(\cdot)=u(\cdot)^{\pm}$. It is well known that
\begin{align*}
    u^{\pm} \in W^{1,p}(\Omega), \quad |u|=u^++u^-, \quad u=u^+-u^-.
\end{align*}
By $|\cdot|$ we denote the Lebesgue measure on $\R^N$. 

The following proposition will be useful in our treatment and was proven in Winkert \cite[Proposition 2.1]{Winkert-2014}

\begin{proposition}\label{proposition_boundary_integral}
    Let $\Omega \subset \R^N$, $N>1,$ be a bounded domain with Lipschitz boundary $\rand$, let $1<p<\infty$, and let $\hat{q}$ be such that $p \leq \hat{q}<p_*$ with the critical exponent stated in \eqref{critical_exponent_boundary}.
    Then, for every $\eps>0$, there exist constants $\tilde{c}_1>0$ and $\tilde{c}_2>0$ such that
    \begin{align*}
	\|u\|_{\hat{q},\rand}^p \leq \eps \|u\|_{1,p}^p+\tilde{c}_1\eps^{-\tilde{c}_2} \|u\|_{p}^p \qquad \text{for all }u \in W^{1,p}(\Omega).
    \end{align*}
\end{proposition}

The next proposition is a standard argument in the application of the Moser iteration, see for example Dr{\'a}bek-Kufner-Nicolosi \cite{Drabek-Kufner-Nicolosi-1997}.

\begin{proposition}\label{proposition_final_interation}
    Let $\Omega \subset \R^N$, $N>1,$ be a bounded domain with Lipschitz boundary $\rand$. Let $u \in \Lp{p}$ with $u\geq 0$ and $1<p<\infty$ such that
    \begin{align}\label{estimate_final}
	\|u\|_{\alpha_n}\leq C
    \end{align}
    with a constant $C>0$ and a sequence $(\alpha_n)\subseteq \R_+$ with $\alpha_n\to \infty$ as $n\to \infty$. Then, $u \in \Linf$.
\end{proposition}

\begin{proof}
    Let us suppose that $u \not\in \Linf$. Then there exist a number $\eta>0$ and a set $A$ of positive measure in $\Omega$ such that $u(x) \geq C+\eta$ for $x\in A$. Then it follows
    \begin{align*}
	\|u\|_{\alpha_n} \geq \left(\int_A u^{\alpha_n}dx\right)^\frac{1}{\alpha_n} \geq \left(C+\eta\right)|A|^{\frac{1}{\alpha_n}}.
    \end{align*}
    Passing to the limit inferior in the inequality above gives
    \begin{align*}
	\liminf_{n\to\infty} \|u\|_{\alpha_n} \geq C+\eta,
    \end{align*}
    which is a contradiction to \eqref{estimate_final}. Hence, $u\in \Linf$.
\end{proof}

\begin{remark}\label{remark_final_iteration}
    It is clear that the statement in Proposition \ref{proposition_final_interation} remains true if we replace the domain $\Omega$ by its boundary $\partial\Omega$.
\end{remark}

Finally, we state a result that the boundedness of a Sobolev function in $\Wp{p}$ implies the boundedness on the boundary.

\begin{proposition}\label{proposition_boundedness_boundary}
    Let $\Omega \subset \R^N$, $N>1,$ be a bounded domain with Lipschitz boundary $\rand$ and let $1<p<\infty$. 
    If $u \in \Wp{p}\cap \Linf$, then $u \in L^\infty(\partial\Omega)$.
\end{proposition}

\begin{proof}
    By the Sobolev embedding we have
    \begin{align*}
	\|v\|_{p_*,\partial\Omega} \leq c_{\partial\Omega} \|v\|_{1,p}\quad\text{for all }v\in \Wp{p}
    \end{align*}
    with the critical exponent $p_*$ as in \eqref{critical_exponent_boundary}. Let $\kappa>1$ and take $v=u^\kappa$ in the inequality above. Note that $v\in \Wp{p}$ since $u\in\Wp{p}\cap \Linf$. This gives
    \begin{align*}
	\|u\|_{\kappa p_*,\partial\Omega} 
	&\leq c_{\partial\Omega}^{\frac{1}{\kappa}} \left[\left(\into |\nabla u^\kappa|^pdx\right)^{\frac{1}{\kappa p}}+\left(\into |u^\kappa|^pdx\right)^{\frac{1}{\kappa p}}\right]\\
	& \leq c_{\partial\Omega}^{\frac{1}{\kappa}} \left[\kappa^{\frac{1}{\kappa}}\|u\|_\infty^{1-\frac{1}{\kappa}} \|\nabla u\|_p^{\frac{1}{\kappa}} + \|u\|_\infty |\Omega|^{\frac{1}{\kappa p}} \right].
    \end{align*}
    Letting $\kappa \to \infty$, by applying Proposition \ref{proposition_final_interation} and Remark \ref{remark_final_iteration}, we derive
    \begin{align*}
	\|u\|_{\infty,\partial\Omega} 
	& \leq 2\|u\|_\infty.
    \end{align*}
\end{proof}

\section{A priori bounds via Moser iteration}\label{section_3}

In this section we state and prove our main result. First, we give the structure conditions on the functions involved in problem \eqref{problem}.

\begin{enumerate}
    \item[(H)]
	The functions $\mathcal{A}: \Omega \times \R \times \R^N \to \R^N$, $\mathcal{B}: \Omega \times \R \times \R^N \to \R$, and $\mathcal{C}: \partial \Omega \times \R \to \R$ are Carath\'eodory functions satisfying the following structure conditions:
	\begin{align*}
	    \text{(H1) \quad } & |\mathcal{A}(x,s,\xi)| \leq a_1|\xi|^{p-1}+a_2|s|^{q_1\frac{p-1}{p}}+a_3, && \text{ for a.a.\,} x\in \Omega,\\
	    \text{(H2) \quad } & \mathcal{A}(x,s,\xi) \cdot \xi \geq a_4|\xi|^{p}-a_5|s|^{q_1}-a_6, && \text{ for a.a.\,} x\in \Omega,\\
	    \text{(H3) \quad } & |\mathcal{B}(x,s,\xi)| \leq b_1|\xi|^{p\frac{q_1-1}{q_1}}+b_2|s|^{q_1-1}+b_3, &&\text{ for a.a.\,} x\in \Omega,\\
	    \text{(H4) \quad } & |\mathcal{C}(x,s)| \leq c_1|s|^{q_2-1}+c_2, && \text{ for a.a.\,} x\in \partial \Omega,
	\end{align*}
	for all $s \in \R$, for all $\xi \in \R^N$, with positive constants $a_i, b_j, c_k$ $\left(i\in \{1,\ldots,6 \},\right.$ $\left.j \in \{1,2,3\}, k\in\{1,2\}\right)$ and fixed numbers $p, q_1, q_2 $ such that
	\begin{align*}
	  1<p<\infty, \qquad p\leq q_1 \leq p^*, \qquad p\leq  q_2 \leq p_*
	\end{align*}
	with the critical exponents stated in \eqref{critical_exponent_domain} and \eqref{critical_exponent_boundary}.
\end{enumerate}

A function $u \in \Wp{p}$ is said to be a weak solution of equation \eqref{problem} if
\begin{align}\label{weak_solution}
    \into \mathcal{A}(x,u,\nabla u) \cdot \nabla \ph dx = \into \mathcal{B}(x,u,\nabla u) \ph dx + \int_{\partial \Omega} \mathcal{C}(x,u) \ph d \sigma
\end{align}
holds for all test functions $\ph \in \Wp{p}$.

By means of the embeddings $i:\Wp{p}\to \Lp{p^*}$ and $\gamma:\Wp{p}\to \Lprand{p_*}$ we see that the definition of a weak solution is well-defined and all integrals in \eqref{weak_solution} are finite for $u,\ph\in\Wp{p}$.

Now we can formulate the main result of our paper.

\begin{theorem}\label{main_theorem}
    Let $\Omega \subset \R^N$, $N>1,$ be a bounded domain with Lipschitz boundary $\rand$ and let the hypotheses (H) be satisfied. Then, every weak solution $u\in\Wp{p}$ of problem \eqref{problem} belongs to $ L^r(\close) $ for every $ r< \infty $. Moreover, 
    $ u \in L^\infty(\close)$, that is, $\|u\|_{\infty}\leq M$, where $M$ is a constant which depends on the given data and on $u $. 
    

\end{theorem}

\begin{proof}
    Let $u\in \Wp{p}$ be a weak solution of problem \eqref{problem}. Since $u=u^+ -u^-$ we can suppose, without any loss of generality, that $u \geq 0$. Furthermore, we only prove the case when $q_1=p^*$ and $q_2=p_*$. The other cases were already obtained in \cite[Theorem 4.1]{Winkert-2010c} and \cite[Theorem 3.1]{Winkert-2014}. Moreover, we will denote positive constants with $M_i$ and if the constant depends on the parameter $\kappa$ we write $M_i(\kappa)$ for $i=1,2,\ldots$.
    
    Let $h>0$ and set $u_h=\min\{u,h\}$. Then we choose $\ph=uu_h^{\kappa p}$ with $\kappa>0$ as test function in \eqref{weak_solution}. Note that $\nabla \ph=\nabla u u_h^{\kappa p}+u\kappa p u_h^{\kappa p-1}\nabla u_h$. This gives
    \begin{align}\label{1}
      \begin{split}
	&\into \mathcal{A}(x,u,\nabla u)\cdot \nabla u u_h^{\kappa p}dx+ \kappa p\int_\Omega \mathcal{A}(x,u,\nabla u) \cdot \nabla u_h u_h^{\kappa p-1}udx\\ 
	&= \into \mathcal{B}(x,u,\nabla u) uu_h^{\kappa p} dx + \int_{\partial \Omega} \mathcal{C}(x,u) uu_h^{\kappa p} d \sigma.
      \end{split}
    \end{align}
    Applying (H2) to the first term of the left-hand side of \eqref{1} yields
    \begin{align}\label{2}
      \begin{split}
	&\into \mathcal{A}(x,u,\nabla u) \cdot \nabla u u_h^{\kappa p}dx\\
	&\geq \int_{\Omega} \left[a_4 |\nabla u|^p-a_5u^{p^*}-a_6\right]u_h^{\kappa p} dx\\
	& \geq a_4\into \left|\nabla u\right|^p u_h^{\kappa p}dx -(a_5+a_6) \into u^{p^*}u_h^{\kappa p}dx-a_6|\Omega|,
      \end{split}
    \end{align}
    respectively to the second term on the left-hand side
    \begin{align}\label{2b}
      \begin{split}
	&\kappa p\int_\Omega \mathcal{A}(x,u,\nabla u) \cdot \nabla u_h u_h^{\kappa p-1}udx\\
	&=\kappa p\int_{\left\{x\in\Omega: u(x)\leq h\right\}} \mathcal{A}(x,u,\nabla u) \cdot \nabla u u_h^{\kappa p}dx\\
	&\geq \kappa p \int_{\left\{x\in\Omega: u(x)\leq h\right\}} \left[a_4 |\nabla u|^p-a_5u^{p^*}-a_6\right]u_h^{\kappa p} dx\\
	& \geq a_4 \kappa p \int_{\left \{x \in \Omega: \, u(x) \le h\right \}} \left|\nabla u \right|^p u_h^{\kappa p}dx -\kappa p (a_5+a_6) \into u^{p^*}u_h^{\kappa p}dx-\kappa pa_6|\Omega|.
      \end{split}
    \end{align}
    By means of (H3) combined with Young's inequality with $\eps_1 >0$, the first term on the right-hand side of \eqref{1} can be estimated through
    \begin{align}\label{3}
	\begin{split}
	    & \into \mathcal{B}(x,u,\nabla u) uu_h^{\kappa p} dx\\
	    & \leq b_1 \into \eps_1^{\frac{p^*-1}{p^*}}|\nabla u|^{p\frac{p^*-1}{p^*}} u_h^{\kappa p \frac{p^*-1}{p^*}} \eps_1^{-\frac{p^*-1}{p^*}}u_h^{\kappa p \left(1-\frac{p^*-1}{p^*} \right)}udx\\
	    & \qquad +(b_2+b_3) \into u^{p^*}u_h^{\kappa p} dx + b_3|\Omega|\\
	    & \leq \eps_1 b_1 \into  |\nabla u|^{p}u_h^{\kappa p} dx +\left(b_1\eps_1^{-(p^*-1)}+b_2+b_3\right)
	    \into u^{p^*}u_h^{\kappa p} dx + b_3|\Omega|.
	\end{split}
    \end{align}
Finally, the boundary term can be estimated via (H4). This leads to
\begin{align}\label{4}
	\begin{split}
	    \int_{\partial \Omega} \mathcal{C}(x,u) uu_h^{\kappa p} d \sigma
	    & \leq \int_{\rand} \left(c_1u^{p_*-1}+c_2\right) uu_h^{\kappa p}d\sigma\\
	    & \leq (c_1+c_2) \int_{\partial\Omega} u^{p_*}u_h^{\kappa p}d\sigma +c_2|\partial \Omega|.
	\end{split}
    \end{align}
    We now combine \eqref{1}-\eqref{4} and choose $ \eps_1= \frac{a_4}{2 b_1} $ to obtain
\begin{equation}\label{5a}
    \begin{split}
	& a_4 \left(\frac{1}{2} \int_{\Omega} \vert \nabla u \vert^p u_h^{\kappa p} dx +  \kappa p \int_{\left\{x \in \Omega: \, u(x) \le h\right\}} \vert \nabla u \vert^p u_h^{\kappa p} dx \right) \\
	& \le \left((\kappa p+ 1)(a_5+ a_6)+ b_1 \eps_1^{-(p^*-1)}+ b_2+ b_3 \right) \int_{\Omega} u^{p^*} u_h^{\kappa p} dx \\
	& \qquad +(c_1+ c_2) \int_{\partial \Omega} u^{p_*} u_h^{\kappa p} d\sigma+ ((\kappa p+1)a_6+ b_3) \vert \Omega \vert+ c_2 \vert \rand \vert.
\end{split}
\end{equation}
Observe that 
\begin{equation*}\label{5b}
    \begin{split}
	& \frac{1}{2} \int_{\Omega}  \vert \nabla u \vert^p u_h^{\kappa p}  dx + \kappa p \int_{\{x \in \Omega: \, u(x) \le h\}} \vert \nabla u \vert^p u_h^{\kappa p} dx \\
	&= \frac{1}{2} \int_{\left\{x \in \Omega: \, u(x)> h\right\}} \vert \nabla u \vert^p u_h^{\kappa p} dx+ \left(\kappa p+ \frac{1}{2}\right) \int_{\left\{x \in \Omega: \, u(x) \le h\right\}} \vert \nabla u \vert^p u_h^{\kappa p} dx \\ 
	& \ge \frac{\kappa p+ 1}{2(\kappa+ 1)^p} \int_{\left\{x \in \Omega: \, u(x)> h\right\}} \vert \nabla u \vert^p u_h^{\kappa p} dx+ \frac{\kappa p+1}{2}  \int_{\left\{x \in \Omega: \, u(x) \le h\right\}} \vert \nabla u \vert^p u_h^{\kappa p} dx \\
	&\ge \frac{\kappa p+ 1}{2(\kappa+ 1)^p} \int_{\Omega} | \nabla \left(u u_h^{\kappa}\right) |^p dx
    \end{split}
\end{equation*}
invoking Bernoulli's inequality $ (\kappa+ 1)^p \ge \kappa p+ 1$. From \eqref{5a} it follows
    \begin{align}\label{5}
      \begin{split}
	&a_4\frac{\kappa p+ 1}{2 (\kappa+1)^p}\into \left|\nabla \left(uu_h^{\kappa}\right)\right|^pdx\\
	& \leq \left((\kappa p+1) (a_5+a_6)+b_1\eps_1^{-(p^*-1)}+b_2+b_3\right )\into u^{p^*}u_h^{\kappa p}dx\\
	&\qquad +(c_1+c_2) \int_{\partial\Omega} u^{p_*}u_h^{\kappa p}d\sigma+\left((\kappa p+ 1) a_6 +b_3 \right) |\Omega|+c_2|\partial \Omega|.
      \end{split}
    \end{align}
    Dividing by $a_4$, summarizing the constants and adding on both sides of \eqref{5} the nonnegative term $\frac{\kappa p+ 1}{(\kappa+1)^p}\|uu_h^{\kappa}\|_{p}^p$ gives
    \begin{align}\label{6}
      \begin{split}
	&\frac{\kappa p+ 1}{(\kappa+1)^p}\|uu_h^{\kappa}\|_{1,p}^p \\
	&\leq \frac{\kappa p+ 1}{(\kappa+1)^p}\|uu_h^{\kappa}\|_{p}^p +M_1(\kappa p+ 1)\into u^{p^*}u_h^{\kappa p}dx
	+M_2\int_{\partial\Omega} u^{p_*}u_h^{\kappa p}d\sigma +M_3\kappa.
      \end{split}
    \end{align}

    {\bf Part I: $u\in L^r(\close)$ for any finite $r$}
    
    Let us now estimate the terms on the right-hand side involving the critical exponents. We set $a:=u^{p^*-p}$ and $b:=u^{p_*-p}$. Moreover, let $L>0$ and $G>0$. Then, by using H\"older's inequality and the Sobolev embeddings for $p^*$ and $p_*$, see Section \ref{section_2}, we get
    \begin{align}\label{7}
      \begin{split}
	&\into u^{p^*}u_h^{\kappa p}dx\\
	&= \int_{\left\{x \in \Omega:\, a(x) \le L\right\}} a u^pu_h^{\kappa p} dx
	+ \int_{\left\{x \in \Omega: \, a(x)> L\right\}} a u^pu_h^{\kappa p} dx \\
	& \le L \int_{\left\{x \in \Omega: \, a(x) \le L\right\}} u^pu_h^{\kappa p} dx\\
	& \qquad+ \left(\int_{\left\{x \in \Omega: \, a(x)> L\right\}} a^{\frac{p^*}{p^*-p}} dx \right)^{\frac{p^*-p}{p^*}} \left(\int_{\Omega} u^{p^*}u_h^{\kappa p^*} dx \right)^{\frac{p}{p^*}} \\
	&\leq L \| uu_h^{\kappa} \|_{p}^p+ \left(\int_{\left\{x \in \Omega: \, a(x)> L\right\}} a^{\frac{p^*}{p^*-p}} dx \right)^{\frac{p^*-p}{p^*}} c_\Omega^p\| uu_h^{\kappa} \|_{1,p}^p
      \end{split}
    \end{align}
    and
    \begin{align}\label{8}
	\begin{split}
	    & \int_{\partial \Omega} u^{p_*}u_h^{\kappa p} d\sigma\\
	    &= \int_{\left\{x \in \partial \Omega: \, b(x) \le G\right\}} b u^pu_h^{\kappa p} d\sigma+ \int_{\left\{x \in \partial\Omega: \, b(x)> G\right\}} b u^pu_h^{\kappa p} d\sigma \\
	    & \le G \int_{\left\{x \in \partial\Omega: \, b(x) \le G\right\}} u^pu_h^{\kappa p} d\sigma\\
	    &\qquad + \left(\int_{\left\{x \in \partial\Omega: \, b(x)> G\right\}} b^{\frac{p_*}{p_*-p}} d\sigma \right)^{\frac{p_*-p}{p_*}} \left(\int_{\partial\Omega} u^{p_*}u_h^{\kappa p_*} d\sigma \right)^{\frac{p}{p_*}} \\
	    &\leq G \| u u_h^{\kappa} \|_{p,\partial\Omega}^p+ \left(\int_{\left\{x \in \partial\Omega: \, b(x)> G\right\}} b^{\frac{p_*}{p_*-p}} d\sigma \right)^{\frac{p_*-p}{p_*}} c_{\partial\Omega}^p\| uu_h^{\kappa} \|_{1,p}^p
	\end{split}
    \end{align}
    with the embedding constants $c_\Omega$ and $c_{\partial\Omega}$. Note that
    \begin{align}\label{9}
      \begin{split}
	H(L)& :=\left(\int_{\left\{x \in \Omega: \, a(x)> L\right\}} a^{\frac{p^*}{p^*-p}} dx \right)^{\frac{p^*-p}{p^*}} \to 0 \quad\text{as }L \to \infty,\\
	K(G)&:=\left(\int_{\left\{x \in \partial\Omega: \, b(x)> G\right\}} b^{\frac{p_*}{p_*-p}} d\sigma \right)^{\frac{p_*-p}{p_*}}\to 0 \quad\text{as }G \to \infty.
      \end{split}
    \end{align}
    Combining \eqref{6}, \eqref{7}, \eqref{8} and \eqref{9} finally yields
    \begin{align}\label{10}
      \begin{split}
	&\frac{\kappa p+ 1}{(\kappa+1)^p}\|uu_h^{\kappa}\|_{1,p}^p \\
	&\leq \left[\frac{\kappa p+1}{(\kappa+1)^p}+M_1(\kappa p+ 1) L\right]\|uu_h^{\kappa}\|_{p}^p +M_1(\kappa p+1) H(L) c_\Omega^p\| uu_h^{\kappa} \|_{1,p}^p\\
	& \qquad  +M_2G \| uu_h^{\kappa} \|_{p,\partial\Omega}^p+ M_2K(G) c_{\partial\Omega}^p\| uu_h^{\kappa} \|_{1,p}^p +M_3\kappa.
      \end{split}
    \end{align}
    
    Now we choose $L=L(\kappa, u)>0$ and $G=G(\kappa, u)>0$ such that
    \begin{align*}
      \begin{split}
	M_1(\kappa p+1) H(L) c_\Omega^p&=\frac{\kappa p+1}{4(\kappa+1)^p},\qquad
	M_2K(G) c_{\partial\Omega}^p=\frac{\kappa p+ 1}{4(\kappa+1)^p}.
      \end{split}
    \end{align*}
    Then, \eqref{10} becomes
    \begin{align}\label{12}
      \begin{split}
	&\frac{\kappa p+1}{2(\kappa+1)^p}\|u u_h^{\kappa}\|_{1,p}^p \\
	&\leq \left[\frac{\kappa p+1}{(\kappa+1)^p}+M_1(\kappa p+ 1) L(\kappa, u)\right]\|u u_h^{\kappa}\|_{p}^p +M_2G(\kappa, u) \| u u_h^{\kappa} \|_{p,\partial\Omega}^p +M_3\kappa,
      \end{split}
    \end{align}
    where $L(\kappa, u)$ and $G(\kappa, u)$ depend on $\kappa$ and on the solution $u$.
    
    {\bf Case I.1: $u \in \Lp{r}$ for any finite $r$}
    
    We can use Proposition \ref{proposition_boundary_integral} to estimate the remaining boundary term in form of
    \begin{align}\label{13}
      \begin{split}
	&\| uu_h^{\kappa} \|_{p,\partial\Omega}^p
	\leq \eps_2 \|uu_h^{\kappa}\|_{1,p}^p+\tilde{c}_1\eps_2^{-\tilde{c}_2} \|uu_h^{\kappa}\|_{p}^p.
      \end{split}
    \end{align}
    Choosing $\eps_2=\frac{1}{M_2 G(\kappa, u)}\frac{\kappa p+1}{4(\kappa+1)^p}$ and applying \eqref{13} to \eqref{12} gives
    \begin{align}\label{14}
      \begin{split}
	&\frac{\kappa p+1}{4(\kappa+1)^p}\|uu_h^{\kappa}\|_{1,p}^p \\
	&\leq \left[\frac{\kappa p+1}{(\kappa+1)^p}+M_1(\kappa p+1)L(\kappa, u)+M_2 G(\kappa, u)\tilde{c}_1\eps_2^{-\tilde{c}_2}\right]\|uu_h^{\kappa}\|_{p}^p +M_3\kappa.
      \end{split}
    \end{align}
    Inequality \eqref{14} can be rewritten as
    \begin{align}\label{15}
      \begin{split}
	&\|uu_h^{\kappa}\|_{1,p}^p \leq M_4(\kappa, u)\left[\|uu_h^{\kappa}\|_{p}^p +1\right]
      \end{split}
    \end{align}
    with a constant $M_4(\kappa, u)$ depending on $\kappa$ and on the function $u$. We may apply the Sobolev embedding theorem on the left-hand side of \eqref{15} which leads to
    \begin{align}\label{81}
	\begin{split}
	    &\|uu_h^\kappa\|_{p^*}\leq c_\Omega \|uu_h^{\kappa}\|_{1,p}
	    \leq M_5(\kappa, u) \left [\|uu_h^{\kappa}\|_{p}^p+1\right]^{\frac{1}{p}}.
	\end{split}
    \end{align}
    Now we can start with the typical bootstrap arguments. We choose $\kappa_1$ such that $(\kappa_1+1)p = p^*$. Then \eqref{81} becomes 
    \begin{align}\label{82}
	\begin{split}
	    \|uu_h^{\kappa_1}\|_{p^*} 
	    &\leq M_5(\kappa_1, u) \left [\|uu_h^{\kappa_1}\|_{p}^p+1\right]^{\frac{1}{p}}  \leq M_6(\kappa_1, u) \left [\|u^{\kappa_1+1}\|_{p}^p+1\right]^{\frac{1}{p}}\\
	    & = M_6(\kappa_1, u) \left [\|u\|_{p^*}^{p^*}+1\right]^{\frac{1}{p}}<\infty,
	\end{split}
    \end{align}
    since $u_h(x)=\min(u(x),h(x)) \leq u(x)$ for a.\,a.\,$x\in \Omega$. Now we may apply Fatou's Lemma as $h\to \infty$ in \eqref{82}. This gives
    \begin{align}\label{83}
	\begin{split}
	    \|u\|_{(\kappa_1+1)p^*} =\|u^{\kappa_1+1}\|_{p^*}^{\frac{1}{\kappa_1+1}}
	    &\leq M_7(\kappa_1, u) \left [\|u\|_{p^*}^{p^*}+1\right]^{\frac{1}{(\kappa_1+1)p}}<\infty.
	\end{split}
    \end{align}
    Hence, $u\in \Lp{(\kappa_1+1)p^*}$. Repeating the steps from \eqref{81}-\eqref{83} for each $\kappa$, we choose a sequence such that
    \begin{align*}
      \begin{split}
	& \kappa_2: (\kappa_2+1)p = (\kappa_1+1) p^*,\\
	& \kappa_3: (\kappa_3+1)p = (\kappa_2+1) p^*,\\
	& \qquad \vdots \qquad \qquad \qquad \qquad \vdots \quad \, .
      \end{split}
    \end{align*}
    This shows that
    \begin{align}\label{16}
        \|u \|_{(\kappa+1)p^*}\leq M_8(\kappa, u)
    \end{align}
    for any finite number $\kappa$, where $M_8(\kappa, u)$ is a positive constant depending both on $\kappa$ and on the solution $u$. Thus, $u \in L^r(\Omega)$ for any $ r \in (1,\infty)$. This proves Case I.1.
    
    {\bf Case I.2: $u \in \Lprand{r}$ for any finite $r$}
    
 Let us repeat inequality \eqref{12} which says
 \begin{align}\label{17}
      \begin{split}
	&\frac{\kappa p+1}{2(\kappa+1)^p}\|u u_h^{\kappa}\|_{1,p}^p \\
	&\leq \left[\frac{\kappa p+1}{(\kappa+1)^p}+M_9(\kappa p +1)L(\kappa, u)\right]\|u u_h^{\kappa}\|_{p}^p\\[1ex] 
	& \qquad +M_{10}G(\kappa, u) \| u u_h^{\kappa} \|_{p,\partial\Omega}^p +M_{11}\kappa.
      \end{split}
    \end{align}
Taking into account \eqref{16}, we can write \eqref{17} in the form
    \begin{align}\label{18}
      \begin{split}
	&\|u u_h^{\kappa}\|_{1,p}^p
	\leq M_{12}(\kappa, u)\left[\| u u_h^{\kappa}\|_{p,\partial\Omega}^p+1\right].
      \end{split}
    \end{align}
Now we may apply the Sobolev embedding theorem for the boundary on the left-hand side of \eqref{18}. This gives
\begin{align}\label{85}
	\begin{split}
	    &\|uu_h^\kappa\|_{p_*, \rand}\leq c_{\partial \Omega} \|uu_h^{\kappa}\|_{1,p}
	    \leq M_{13}(\kappa, u)\left[\| u u_h^{\kappa}\|_{p,\partial\Omega}^p+1\right]^{\frac{1}{p}}.
	\end{split}
\end{align}
As before we proceed with a bootstrap argument and choose $\kappa_1$ in \eqref{85} such that $(\kappa_1+1)p = p_*$. This yields
\begin{align}\label{86}
	\begin{split}
	    \|uu_h^{\kappa_1}\|_{p_*, \rand}
	    &\leq M_{13}(\kappa_1, u)\left[\| u u_h^{\kappa_1}\|_{p,\partial\Omega}^p+1\right]^{\frac{1}{p}}\leq M_{14}(\kappa_1, u)\left[\| u^{\kappa_1+1}\|_{p,\partial\Omega}^p+1\right]^{\frac{1}{p}}\\
	    &\leq M_{14}(\kappa_1, u)\left[\| u\|_{p_*,\partial\Omega}^{p_*}+1\right]^{\frac{1}{p}}<\infty.
	\end{split}
    \end{align}
Applying again Fatou's Lemma we obtain from \eqref{86}
\begin{align}\label{87}
	\begin{split}
	    \|u\|_{(\kappa_1+1)p_*,\partial\Omega} =\|u^{\kappa_1+1}\|_{p_*,\partial\Omega}^{\frac{1}{\kappa_1+1}}
	    &\leq M_{15}(\kappa_1, u) \left [\|u\|_{p_*,\partial\Omega}^{p_*}+1\right]^{\frac{1}{(\kappa_1+1)p}}<\infty.
	\end{split}
    \end{align}
Therefore, $u\in \Lprand{(\kappa_1+1)p_*}$. For each $\kappa$ we repeat the steps from \eqref{85}--\eqref{87} and choose a sequence such that
    \begin{align*}
      \begin{split}
	& \kappa_2: (\kappa_2+1)p = (\kappa_1+1) p_*,\\
	& \kappa_3: (\kappa_3+1)p = (\kappa_2+1) p_*,\\
	& \qquad \vdots \qquad \qquad \qquad \qquad \vdots \quad \, .
      \end{split}
    \end{align*}
    We obtain
    \begin{align}\label{19}
        \|u \|_{(\kappa+1)p_*,\partial\Omega}\leq M_{16}(\kappa, u)
    \end{align}
    for any finite number $\kappa$, where $M_{16}(\kappa, u)$ is a positive constant depending on $\kappa$ and on the solution $u $. Thus, $u \in L^r(\partial \Omega)$ for any $ r \in (1,\infty)$, and therefore $u \in L^r(\close)$ for any finite $r\in (1,\infty)$. This completes the proof of Part I.
    
    {\bf Part II: $u\in L^\infty(\close)$}
    
   Let us recall inequality \eqref{6} which says
    \begin{align}\label{20}
      \begin{split}
	&\frac{\kappa p+1}{(\kappa+1)^p}\|uu_h^{\kappa}\|_{1,p}^p \\
	&\leq \frac{\kappa p+1}{(\kappa+1)^p}\|uu_h^{\kappa}\|_{p}^p +M_{17}(\kappa p+1)\into u^{p^*}u_h^{\kappa p}dx\\[1ex]
	& \qquad +M_{18}\int_{\partial\Omega} u^{p_*}u_h^{\kappa p}d\sigma +M_{19}\kappa.
      \end{split}
    \end{align}
    Let us fix numbers $\tilde{q}_1\in (p,p^*)$ and $\tilde{q}_2\in (p,p_*)$. Then, by applying H\"older's inequality and the results of Part I, see \eqref{16} and \eqref{19}, we derive for the several terms on the right-hand side of \eqref{20}
    \begin{align}\label{21}
      \begin{split}
	\|uu_h^{\kappa}\|_{p}^p
	&\leq |\Omega|^{\frac{\tilde{q}_1-p}{\tilde{q}_1}} \left(\into (uu_h^{\kappa})^{\tilde{q}_1}dx\right)^{\frac{p}{\tilde{q}_1}}\leq M_{20}\|uu_h^{\kappa}\|_{\tilde{q}_1}^p,\\
	\into u^{p^*}u_h^{\kappa p}dx
	&=\into u^{p^*-p}(uu_h^{\kappa})^pdx\\ 
	& \leq \left(\into u^{\frac{p^*-p}{\tilde{q}_1-p}\tilde{q}_1}dx\right)^{\frac{\tilde{q}_1-p}{\tilde{q}_1}}\left(\into (uu_h^{\kappa})^{\tilde{q}_1}\right)^{\frac{p}{\tilde{q}_1}} \leq M_{21}\|uu_h^{\kappa}\|_{\tilde{q}_1}^p,\\
	\int_{\partial\Omega} u^{p_*}u_h^{\kappa p}d\sigma
	&=\int_{\partial\Omega} u^{p_*-p}(uu_h^{\kappa})^pd\sigma\\
	&\leq \left(\int_{\partial\Omega} 
	u^{\frac{p_*-p}{\tilde{q}_2-p}\tilde{q}_2}d\sigma\right)
	^{\frac{\tilde{q}_2-p}{\tilde{q}_2}}\left(\int_{\partial\Omega} (uu_h^{\kappa})^{\tilde{q}_2}d\sigma\right)^{\frac{p}{\tilde{q}_2}}\\
	& \leq M_{22}\|uu_h^{\kappa}\|_{\tilde{q}_2,\partial\Omega}^p.
      \end{split}
    \end{align}
    Note that $M_{21}, M_{22}$ are finite because of Part I. Moreover, we see from the calculations above that
    \begin{align}\label{constants}
	M_{21}=M_{21}\left(\|u\|_{\frac{p^*-p}{\tilde{q}_1-p}\tilde{q}_1}\right) \quad\text{and}\quad M_{22}=M_{22}\left(\|u\|_{\frac{p_*-p}{\tilde{q}_2-p}\tilde{q}_2}\right).
    \end{align}
    Using \eqref{21} to \eqref{20} leads to
    \begin{align}\label{22}
    \begin{split}
	&\frac{\kappa p+1}{(\kappa+1)^p}\|uu_h^{\kappa}\|_{1,p}^p \leq M_{23}\frac{\kappa p+1}{(\kappa+1)^p}\|uu_h^{\kappa}\|_{\tilde{q}_1}^p+M_{24}\|uu_h^{\kappa}\|_{\tilde{q}_2,\partial\Omega}^p +M_{25}\kappa.
      \end{split}
    \end{align}
    
    {\bf Case II.1: $u\in \Linf$}
    
    As before, we can estimate the boundary term via Proposition \ref{proposition_boundary_integral} and then use H\"older's inequality as seen in the first line of \eqref{21}. This gives
    \begin{align}\label{23}
      \begin{split}
	\|uu_h^{\kappa}\|_{\tilde{q}_2,\partial\Omega}^p
	&\leq \eps_3 \|uu_h^{\kappa}\|_{1,p}^p+\tilde{c}_1\eps_3^{-\tilde{c}_2} \|uu_h^{\kappa}\|_{p}^p\\
	& \leq \eps_3 \|uu_h^{\kappa}\|_{1,p}^p+\tilde{c}_1\eps_3^{-\tilde{c}_2} M_{20}\|uu_h^{\kappa}\|_{\tilde{q}_1}^p.
      \end{split}
    \end{align}
    Now we choose $\eps_3=\frac{\kappa p+1}{2M_{24}(\kappa+1)^p}$ and apply \eqref{23} in \eqref{22} to obtain
    \begin{align}\label{24}
      \begin{split}
	&\frac{\kappa p+1}{2(\kappa+1)^p}\|uu_h^{\kappa}\|_{1,p}^p 
	\leq \left(M_{23}(\kappa p+1)+\tilde{c}_1\eps_3^{-\tilde{c}_2} M_{20}M_{24}\right)\|uu_h^{\kappa}\|_{\tilde{q}_1}^p+M_{25}\kappa.
      \end{split}
    \end{align}
    Inequality \eqref{24} can be rewritten in the form
    \begin{align}\label{25}
      \begin{split}
	&\|uu_h^{\kappa}\|_{1,p}^p 
	\leq M_{26}\left((\kappa+1)^p\right)^{M_{27}} \left[\|uu_h^{\kappa}\|_{\tilde{q}_1}^p+1\right].
      \end{split}
    \end{align}
    In order so see this, note that
    \begin{align*}
	&\frac{2(\kappa+1)^p}{\kappa p+1}
	\left(M_{23}(\kappa p+1)+\tilde{c}_1\eps_3^{-\tilde{c}_2} M_{20}M_{24}\right)\\
	&= 2(\kappa+1)^p
	\left(M_{23}+\tilde{c}_1\left(\frac{2M_{24}(\kappa+1)^p}{\kappa p+1}\right)^{\tilde{c}_2} \frac{1}{\kappa p+1}
	M_{20}M_{24}\right)\\
	&\leq M_{26}\left((\kappa+1)^p\right)^{M_{27}}.
    \end{align*}
    Now we may apply the Sobolev embedding on the left-hand side of \eqref{25} and the fact that $u \in \Lp{r}$ for any finite $r \in (1,\infty)$ to get
    \begin{align}\label{26}
	\begin{split}
	    \|uu_h^{\kappa}\|_{p^*}
	    &\leq c_\Omega \|uu_h^{\kappa}\|_{1,p} \leq M_{27} \left((\kappa+1)^{M_{28}}\right) \left[\|uu_h^{\kappa}\|_{\tilde{q}_1}^p+1\right]^{\frac{1}{p}}\\
	    & \leq M_{29} \left((\kappa+1)^{M_{28}}\right) \left[\|u^{\kappa+1}\|_{\tilde{q}_1}^p+1\right]^{\frac{1}{p}}<\infty.
	\end{split}
    \end{align}
    Applying Fatou's Lemma in \eqref{26} implies that
    \begin{align}\label{26b}
	\begin{split}
	    \|u\|_{(\kappa+1)p^*}= \|u^{\kappa+1}\|_{p^*}^{{\frac{1}{\kappa+1}}}
	    &\leq M_{29}^{\frac{1}{\kappa+1}} \left((\kappa+1)^{M_{28}}\right)^{{\frac{1}{\kappa+1}}} \left[\|u^{\kappa+1}\|_{\tilde{q}_1}^p+1\right]^{\frac{1}{(\kappa+1)p}}.
	\end{split}
    \end{align}
    Observe that
    \begin{align*}
	\left((\kappa+1)^{M_{28}}\right)^{\frac{1}{\sqrt{\kappa+1}}} \geq 1
	\qquad \text{and} \qquad
	\lim_{\kappa \to \infty}\left( (\kappa+1)^{M_{28}}\right)^{\frac{1}{\sqrt{\kappa+1}}}=1.
    \end{align*}
    Hence, we find a constant $M_{30}>1$ such that
    \begin{align}\label{27}
	\left((\kappa+1)^{M_{28}}\right)^{\frac{1}{\kappa+1}} \leq M_{30}^{\frac{1}{\sqrt{\kappa+1}}}.
    \end{align}
    From \eqref{26b} and \eqref{27} we derive
    \begin{align}\label{28}
	\begin{split}
	    \|u\|_{(\kappa+1)p^*}
	    &\leq M_{29}^{\frac{1}{\kappa+1}} M_{30}^{\frac{1}{\sqrt{\kappa+1}}}\left[\|u^{\kappa+1}\|_{\tilde{q}_1}^p+1\right]^{\frac{1}{(\kappa+1)p}}.
	\end{split}
    \end{align}
    Now we are ready to prove the uniform boundedness with respect to $\kappa$. To this end, suppose there is a sequence $\kappa_n\to \infty$ such that
    \begin{align*}
        \|u^{\kappa_n+1}\|_{\tilde{q}_1}^p \leq 1,
    \end{align*}
    which is equivalent to
    \begin{align*}
	\|u\|_{(\kappa_n+1)\tilde{q}_1}\leq 1,
    \end{align*}
    then Proposition \ref{proposition_final_interation} implies that $\|u\|_{\infty}<\infty$.

    In the opposite case there exists a number $\kappa_0>0$ such that
    \begin{align}\label{29}
        \|u^{\kappa+1}\|_{\tilde{q}_1}^p >1 \qquad \text{for any }\kappa \geq \kappa_0.
    \end{align}
    Combining \eqref{28} and \eqref{29} yields
    \begin{align}\label{30}
	\begin{split}
	    \|u\|_{(\kappa+1)p^*}
	    &\leq M_{29}^{\frac{1}{\kappa+1}} M_{30}^{\frac{1}{\sqrt{\kappa+1}}}\left [2\|u^{\kappa+1}\|_{\tilde{q}_1}^p\right]^{\frac{1}{(\kappa+1)p}}
	    \leq M_{31}^{\frac{1}{\kappa+1}} M_{30}^{\frac{1}{\sqrt{\kappa+1}}} \|u\|_{(\kappa+1)\tilde{q}_1}
	\end{split}
    \end{align}
    for any $\kappa \geq \kappa_0$. Applying again the bootstrap arguments we define a sequence $(\kappa_n)$ such that
    \begin{align}\label{31}
	\begin{split}
	    & \kappa_1: (\kappa_1+1)\tilde{q}_1 = (\kappa_0+1)p^*,\\
	    & \kappa_2: (\kappa_2+1)\tilde{q}_1 = (\kappa_1+1) p^*,\\
	    & \kappa_3: (\kappa_3+1)\tilde{q}_1 = (\kappa_2+1) p^*,\\
	    & \qquad \vdots \qquad \qquad \qquad \qquad \vdots \quad \, .
	\end{split}
    \end{align}
    By induction, from \eqref{30} and \eqref{31}, we obtain
    \begin{align*}
	\begin{split}
	    \|u\|_{(\kappa_n+1)p^*}
	    &\leq M_{31}^{\frac{1}{\kappa_n+1}} M_{30}^{\frac{1}{\sqrt{\kappa_n+1}}}\left\|u\right\|_{(\kappa_n+1)\tilde{q}_1} =M_{31}^{\frac{1}{\kappa_n+1}} M_{30}^{\frac{1}{\sqrt{\kappa_n+1}}}\left\|u\right\|_{(\kappa_{n-1}+1)p^*}
	\end{split}
    \end{align*}
    for any $n \in \N$, where the sequence $(\kappa_n)$ is chosen in such a way that $(\kappa_n+1)=(\kappa_{0}+1)\left(\frac{p^*}{\tilde{q}_1}\right)^n$. Following this we see that
    \begin{align*}
	\begin{split}
	    \|u\|_{(\kappa_n+1)p^*}
	    & \leq  M_{31}^{\sum\limits_{i=1}^n\frac{1}{\kappa_i+1}} M_{30}^{\sum\limits_{i=1}^n \frac{1}{\sqrt{\kappa_i+1}}}\|u \|_{(\kappa_0+1)p^*}
	\end{split}
    \end{align*}
    with $(\kappa_n+1)p^* \to \infty $ as $n \to \infty$. Since $\frac{1}{\kappa_i+1}=\frac{1}{\kappa_0+1}\left(\frac{\tilde{q}_1}{p^*}\right)^i$ and $\frac{\tilde{q}_1}{p^*}<1$, there is a constant $M_{32}>0$ such that
    \begin{align*}
	\begin{split}
	    \|u\|_{(\kappa_n+1)p^*} \leq M_{32}\|u\|_{(k_0+1)p^*}<\infty,
	\end{split}
    \end{align*}
    where the finiteness of the right-hand side follows from Part I. Now we may apply Proposition \ref{proposition_final_interation} to conclude that $u\in \Linf$, that is, there exists $M>0$, which depends on the given data and on $u$, such that $\|u\|_\infty \leq M$.
    
    {\bf Case II.2: $u\in L^\infty(\partial \Omega)$}
    
    This case follows directly from Case II.1 and Proposition \ref{proposition_boundedness_boundary}.

    Combining Case II.1 and II.2 shows that $u \in L^\infty(\close)$.
\end{proof}

\begin{remark}
    It is clear that hypothesis (H1) is not needed in the proof of Theorem \ref{main_theorem}, but it is necessary to have a well-defined definition of a weak solution. 
\end{remark}

\begin{remark}

Since problem \eqref{problem} involves functions that can exhibit a critical growth, one cannot expect to find a constant $ M$ which depends in an explicit way on natural norms such as $ \|u\|_{p^*} $ or $ \|u\|_{p_*, \rand}$. But, if one searches for a dependence on norms that are greater than the critical ones, then a possible dependence is given on the norms $ \Vert u \Vert_{\frac{p^*-p}{\tilde q_1- p}\tilde q_1 } $ as well as $ \Vert u \Vert_{\frac{p_*-p}{\tilde q_2-p} \tilde q_2, \rand} $, where $ \tilde q_1 \in (p, p^*) $ and $ \tilde q_2 \in (p, p_*) $, as seen in the proof of Theorem \ref{main_theorem}, see \eqref{constants}. 

\end{remark}

Based on the results of Theorem \ref{main_theorem}, we obtain regularity results for solutions of type \eqref{problem}. For simplification we drop the $s$-dependence of the operator. To this end, let $\vartheta \in C^1(0,\infty)$ be a function such that
\begin{align}\label{101}
    0 < a_1  \leq \frac{t \vartheta'(t)}{\vartheta(t)} \leq a_2 \quad \text{ and } \quad a_3 t^{p-1} \leq \vartheta(t) \leq a_4 \l(1+t^{p-1}\r)
\end{align}
for all $t>0$, with some constants $a_i>0$, $i\in\{1,2,3,4\}$ and for $1 <p<\infty$. The hypotheses on $\mathcal{A}: \close \times \R^N \to \R^N$ read as follows.

\begin{enumerate}[leftmargin=1.2cm]
    \item[H($\mathcal{A}$):]
	$\mathcal{A}(x,\xi)=\mathcal{A}_0\l(x,|\xi|\r)\xi$ with $\mathcal{A}_0 \in C(\close \times \R_+)$ for all $\xi \in \R^N$, where $\R_+=[0,+\infty)$ and with $\mathcal{A}_0(x,t)>0$ for all $x\in \close$ and for all $t>0$. Moreover,
	\begin{enumerate}[leftmargin=0.8cm,topsep=0.2cm,itemsep=0.1cm]
	    \item[(i)]
		$\mathcal{A}_0 \in C^1(\close \times (0,\infty))$, $t\to t\mathcal{A}_0(x,t)$ is strictly increasing in $(0,\infty)$, $\displaystyle\lim_{t \to 0^+} t \mathcal{A}_0(x,t)=0$ for all $x\in \close$ and
		\begin{align*}
		    \lim_{t \to 0^+} \frac{t \mathcal{A}_0'(x,t)}{\mathcal{A}_0(x,t)}=c>-1\quad \text{for all }x\in\close;
		\end{align*}
	    \item[(ii)]
		$\displaystyle |\nabla_\xi \mathcal{A}(x,\xi)| \leq a_5 \frac{\vartheta\l(|\xi|\r)}{|\xi|}$ for all $x\in \close$, for all $\xi \in \R^N \setminus \{0\}$ and for some $a_5>0$;
	    \item[(iii)]
		$\displaystyle  \nabla_\xi \mathcal{A}(x,\xi) y \cdot y \geq \frac{\vartheta \l(|\xi|\r)}{|\xi|} |y|^2$ for all $x\in \close$, for all $\xi \in \R^N \setminus\{0\}$ and for all $y \in \R^N$.
	\end{enumerate}
\end{enumerate}

\begin{remark}
    We chose the special structure in H($\mathcal{A}$) to apply the nonlinear re\-gularity theory, which is mainly based on the results of Lieberman \cite{Lieberman-1991} and Pucci-Serrin \cite{Pucci-Serrin-2007}. Closely related to this subject is also the work by Motreanu-Motreanu-Papageorgiou \cite{Motreanu-Motreanu-Papageorgiou-2011}. 
    If we set
    \begin{align*}
	G_0(x,t)=\int_0^t \mathcal{A}_0(x,s)sds,
    \end{align*}
    then $G_0 \in C^1(\close\times \R_+)$ and the function $G_0(x,\cdot)$ is increasing and strictly convex for all $x\in \close$. We set $G(x,\xi)=G_0(x,|\xi|)$ for all $(x,\xi)\in \close \times \R^N$ and obtain that $G\in C^1(\close \times \R^N)$ and that the function $\xi \to G(x,\xi)$ is convex. Moreover, we easily derive that
    \begin{align*}
	\nabla_\xi G(x,\xi)=(G_0)_t'(x,|\xi|)\frac{\xi}{|\xi|}=\mathcal{A}_0(x,|\xi|)\xi=\mathcal{A}(x,\xi)
    \end{align*}
    for all $\xi \in \R^N\setminus \{0\}$ and $\nabla_\xi G(x,0)=0$. So, $G(x,\cdot)$ is the primitive of $\mathcal{A}(x,\cdot)$. This fact, the convexity of $G(x,\cdot)$ and since $G(x,0)=0$ for all $x\in \close$ imply that
    \begin{align}\label{102}
	G(x,\xi)\leq \mathcal{A}(x,\xi)\cdot \xi\quad\text{for all } (x,\xi)\in \close\times \R^N.
    \end{align}
\end{remark}

The next lemma summarizes the main properties of $\mathcal{A}:\close\times \R^N\to \R^N$. The result is an easy consequence of \eqref{101} and the hypotheses H($\mathcal{A}$).

\begin{lemma}
    If hypotheses H($\mathcal{A}$) are satisfied, then the following hold:
    \begin{enumerate}[leftmargin=0.8cm,topsep=0.0cm,itemsep=0.1cm]
	\item[(i)]
	    $\mathcal{A}\in C(\close \times \R^N,\R^N)\cap C^1(\close \times (\R^N \setminus \{0\}),\R^N)$ and the map $\xi \to \mathcal{A}(x,\xi)$ is continuous and strictly monotone (hence, maximal monotone) for all $x\in\close$;
	\item[(ii)]
	    $|\mathcal{A}(x,\xi)| \leq a_6 \l(1+|\xi|^{p-1}\r)$ for all $x\in \close$, for all $\xi \in \R^N$ and for some $a_6>0$;
	\item[(iii)]
	    $\mathcal{A}(x,\xi) \cdot \xi \geq \frac{a_3}{p-1} |\xi|^p$ for all $x \in \close$ and for all $\xi \in \R^N$.
    \end{enumerate}
\end{lemma}

From this lemma along with \eqref{102} we easily deduce the following growth estimates for the primitive $G(x,\cdot)$.

\begin{corollary}
    If hypotheses H($\mathcal{A}$) hold, then
    \begin{align*}
	\frac{a_3}{p(p-1)}|\xi|^p \leq G(x,\xi) \leq a_7\l(1+|\xi|^p\r)
    \end{align*}
    for all $x\in \close$, for all $\xi \in \R^N$ and for some $a_7>0$.
\end{corollary}

Let $A : \Wp{p} \to \Wp{p}^*$ be the nonlinear map defined by
\begin{align}\label{operator_representation}
    \left \langle A(u),\ph \right \ran= \into \mathcal{A}(x,\nabla u)\cdot \nabla \ph dx \quad \text{for all } u,\ph \in \Wp{p}.
\end{align}
The next proposition summarizes the main properties of this operator, see Gasi{\'n}ski-Papageorgiou \cite{Gasinski-Papageorgiou-2014}.

\begin{proposition}
    Let the hypotheses H($\mathcal{A}$) be satisfied and let $A : \Wp{p} \to \Wp{p}^*$ be the map defined in \eqref{operator_representation}. Then, $A$ is bounded, continuous, monotone (hence maximal monotone) and of type $(\Ss_+)$.
\end{proposition}

Let us state some operators which fit in our setting and which are of much interest.
\begin{example}
    For simplicity, we drop the $x$-dependence of the operator $\mathcal{A}$. The following maps satisfy hypotheses H($\mathcal{A}$):
    \begin{enumerate}[leftmargin=0.8cm]
	\item[(i)]
	    Let $\mathcal{A}(\xi)=|\xi|^{p-2}\xi$ with $1<p<\infty$. This map corresponds to the $p$-Laplace differential operator defined by
	    \begin{align*}
		\Delta_p u=\divergenz \left(|\nabla u|^{p-2} \nabla u \right) \quad \text{for all } u \in W^{1,p}(\Omega).
	    \end{align*}
	    The potential is $G(\xi)=\frac{1}{p}|\xi|^p$ for all $\xi \in \R^N$.
	\item[(ii)]
	    The function $\mathcal{A}(\xi)=|\xi|^{p-2}\xi+\mu |\xi|^{q-2}\xi$ with $1<q<p<\infty$ and $\mu>0$ compares with the $(p,q)$-differential operator defined by $\Delta_p u+ \mu \Delta_q u$ for all $u \in W^{1,p}(\Omega)$. The potential is $G(\xi)=\frac{1}{p}|\xi|^p+\frac{\mu}{q}|\xi|^q$ for all $\xi \in \R^N$.
	\item[(iii)]
	    If $\mathcal{A}(\xi)=\left(1+|\xi|^2\right)^{\frac{p-2}{2}}\xi$ with $1<p<\infty$, then this map represents the generalized $p$-mean curvature differential operator defined by
	    \begin{align*}
		\divergenz \left [(1+|\nabla u|^2)^{\frac{p-2}{2}} \nabla u \right] \quad \text{ for all } u \in W^{1,p}(\Omega).
	    \end{align*}
	    The potential is $G(\xi)=\frac{1}{p}\left[(1+|\xi|^2)^{\frac{p}{2}} -1 \right]$ for all $\xi \in \R^N$.
    \end{enumerate}
\end{example}
Let us write hypotheses (H) without the structure conditions on $\mathcal{A}$.

\begin{enumerate}[leftmargin=1.5cm]
    \item[H($\mathcal{B},\mathcal{C}$):]
	The functions $\mathcal{B}: \Omega \times \R \times \R^N \to \R$ and $\mathcal{C}: \partial \Omega \times \R \to \R$ are Carath\'eodory functions satisfying the following structure conditions:
	\begin{align*}
	    |\mathcal{B}(x,s,\xi)| & \leq b_1|\xi|^{p\frac{q_1-1}{q_1}}+b_2|s|^{q_1-1}+b_3,  \qquad \text{ for a.a.\,}   x\in \Omega,\\
	    |\mathcal{C}(x,s)| &\leq c_1|s|^{q_2-1}+c_2, \qquad \text{ for a.a.\,} x\in \rand,
	\end{align*}
	for all $s \in \R$, for all $\xi \in \R^N$, with positive constants $b_j, c_k$ $(j \in \{1,2,3\}$, $k\in\{1,2\})$ and fixed numbers $p, q_1, q_2 $ such that
	\begin{align*}
	  1<p<\infty, \qquad p\leq q_1 \leq p^*, \qquad p\leq q_2 \leq p_*
	\end{align*}
	with the critical exponents stated in \eqref{critical_exponent_domain} and \eqref{critical_exponent_boundary}. Moreover, $\mathcal{C}$ satisfies the condition
	\begin{align*}
	    \left|\mathcal{C}(x,s)-\mathcal{C}(y,t)\right| \leq L \left[|x-y|^{\alpha}+|s-t|^\alpha \right], \quad |\mathcal{C}(x,s)| \leq L
	\end{align*}
	for all $(x,s), (y,t) \in \partial \Omega \times [-M_0,M_0]$ with $\alpha \in (0,1]$ and constants $M_0>0$ and $L \geq 0$.

\end{enumerate}

Based on the hypotheses H($\mathcal{A}$) and H($\mathcal{B},\mathcal{C}$), problem \eqref{problem} becomes
\begin{align}\label{problem2}
    \begin{aligned}
        -\divergenz  \mathcal{A}(x,\nabla u) & = \mathcal{B}(x,u,\nabla u)  & \hspace*{0.5cm} & \text{ in } \Om, \\
    \mathcal{A}(x,\nabla u) \cdot \nu & = \mathcal{C}(x,u)  && \text{ on } \partial \Omega.
    \end{aligned}
\end{align}

Combining Theorem \ref{main_theorem} and the regularity theory of Lieberman \cite{Lieberman-1991} leads to the following result.

\begin{theorem}
  Let $\Omega \subset \R^N$, $N>1,$ be a bounded domain with a $C^{1,\alpha}$-boundary $\rand$ and let the assumptions H($\mathcal{A}$) and H($\mathcal{B}, \mathcal{C}$) be satisfied. Then, every weak solution $u\in\Wp{p}$ of problem \eqref{problem2} belongs to $C^{1,\beta}(\close)$ for some $\beta \in (0,1)$ such that $\beta=\beta(a_1,a_2,a_5, \alpha, N)$ and 
  \begin{equation*}
      \| u \|_{C^{1, \beta}(\close)}  \leq C(a_1, a_2, a_3, a_5, N, \vartheta(1), M, \alpha, b_1, b_2, b_3)
      \end{equation*}
  where $M$ is the constant that comes from the statement of Theorem \ref{main_theorem}.
\end{theorem}

\begin{proof}
    We will apply Theorem 1.7 of Lieberman \cite{Lieberman-1991} and the comment after this theorem concerning global H\"older gradient estimates. First, we know from Theorem \ref{main_theorem} that $\|u\|_\infty \leq M$. The only thing we need to do is to check that the conditions (1.10a)--(1.10d) in \cite[p.\,320]{Lieberman-1991} are satisfied. From conditions H($\mathcal{A}$)(iii), (ii) we see that the assumptions (1.10a) and (1.10b) are satisfied. Moreover, from H($\mathcal{B}, \mathcal{C}$) and \eqref{101} we obtain
    \begin{align*}
	|\mathcal{B} (x, s, \xi) | 
	&\leq b_1 | \xi |^{p \frac{q_1- 1}{q_1}}+ b_2 | s |^{q_1- 1}+ b_3 \\
	& \leq b_1 | \xi |^p+ b_1+ b_2 M^{q_1- 1}+ b_3 \\
	& = b_1 |\xi|^{p-1} |\xi |+ b_1+ b_2 M^{q_1- 1}+ b_3 \\
	& \leq \frac{b_1}{a_3} \vartheta(|\xi|) |\xi |+ b_1+ b_2 M^{q_1- 1}+ b_3 \\
	& \leq\max\left\{\frac{b_1}{a_3}, b_1+ b_2 M^{q_1- 1}+ b_3\right\} \left(\vartheta(|\xi|) | \xi|+ 1\right).
    \end{align*}
    This proves condition (1.10d). Assumption (1.10c) follows from the fact that the function $\mathcal{A}$ is continuous differentiable in the space variable and independent of the $s$-variable. Then we may apply the mean value theorem which shows (1.10c). The desired result follows from Lieberman \cite[Theorem 1.7]{Lieberman-1991} with the constants $\beta$, $C$ as in the theorem (and their dependence on the data) and the constant $M$ from Theorem \ref{main_theorem}.
\end{proof}

\section*{Acknowledgment}

The authors wish to thank the referees for their corrections and insightful remarks that helped to improve the paper.

This work has been performed in the framework of Piano della Ricerca 2016-2018--linea di intervento 2: ``Metodi variazionali ed equazioni differenziali''. The first author was partially supported by Gruppo Nazionale per l'Analisi Matema\-tica, la Probabilit\`a e le loro Applicazioni (INdAM). The second author thanks the University of Catania for the kind hospitality during a research stay in March 2018.

\end{document}